\documentclass[12pt,letterpaper,reqno]{amsart}

\usepackage{mathptmx}
\usepackage{amsmath}
\usepackage{amssymb}
\usepackage{mathtools}
\usepackage{mathrsfs}
\usepackage{bm}
\usepackage{enumitem}
\usepackage{microtype}
\usepackage{xcolor}
\usepackage[
colorlinks=true,
linkcolor=red,
citecolor=blue,
urlcolor=blue,
pdfencoding=auto,
psdextra
]{hyperref}
\usepackage[nameinlink,noabbrev,capitalise]{cleveref}

\newcommand{\Z}{\mathbb Z}
\newcommand{\cA}{\mathcal A}
\newcommand{\cB}{\mathcal B}
\newcommand{\cF}{\mathcal F}
\newcommand{\cP}{\mathcal P}
\newcommand{\cR}{\mathcal R}
\newcommand{\cU}{\mathcal U}
\newcommand{\join}{\bigvee}
\newcommand{\mesh}{\operatorname{mesh}}
\newcommand{\htop}{h_{\mathrm{top}}}
\newcommand{\Nname}{\mathsf N}
\newcommand{\Hb}{H_{\mathrm b}}

\DeclareMathOperator{\diam}{diam}

\newtheorem{theorem}{Theorem}[section]
\newtheorem{lemma}[theorem]{Lemma}
\newtheorem{proposition}[theorem]{Proposition}
\newtheorem{corollary}[theorem]{Corollary}
\newtheorem{question}[theorem]{Question}
\newtheorem{conjecture}[theorem]{Conjecture}

\theoremstyle{definition}
\newtheorem{definition}[theorem]{Definition}

\theoremstyle{remark}
\newtheorem{remark}[theorem]{Remark}

\numberwithin{equation}{section}
\setlist[enumerate,1]{label=\textup{(\roman*)},leftmargin=*,itemsep=2pt}
\setlist[itemize]{leftmargin=*,itemsep=2pt}
\allowdisplaybreaks[3]
\title[Hereditary Lowerability]{Hereditary Lowerability of Topological Dynamical Systems}

\author[X. Wang]{Xiaochen Wang}
\address{School of Mathematical Sciences, University of Science and Technology of China, Hefei, 230026, Anhui, People’s Republic of China}
\email{xcwang97@ustc.edu.cn}

\date{}

\hypersetup{
	pdftitle={Answers to a Question and a Conjecture on Lowering Topological Entropy},
	pdfauthor={Author Name},
	pdfsubject={Topological and measure-theoretic entropy},
	pdfkeywords={topological entropy, lowerable, hereditarily lowerable}
}

\begin{document}
	
	\begin{abstract}
		Let $(X,T)$ be a topological dynamical system and let $h(T,K)$ denote the topological
		entropy of a compact set $K\subset X$. We settle a question and a
		conjecture raised by Huang, Ye, and Zhang concerning hereditary lowerability in \cite{HYZ2014}.
		First, we show that every system with finite topological entropy is hereditarily lowerable:
		for every nonempty compact set $K\subset X$ and every
		$0\leq h\leq h(T,K)$, there is a compact set $K_h\subset K$ such that
		$h(T,K_h)=h$. This gives a negative answer to their Question~$2'$.
		Second, we prove that if $(X,T)$ admits an ergodic invariant measure with infinite entropy, then $(X,T)$ is not hereditarily lowerable. More precisely, we construct a compact set $K$ with infinite entropy such that every compact subset
		of $K$ has entropy either zero or infinity. This proves the conjecture stated
		immediately after Question~$2'$.
	\end{abstract}
	
	\maketitle
	\pagestyle{headings}
	\markboth{}{}
	
	\section{Introduction}
	\label{sec:introduction}
	
	Let $(X,T)$ be a topological dynamical system, by which we mean a compact
	metric space $X$ together with a homeomorphism $T\colon X\to X$. For a compact
	set $K\subset X$, write $h(T,K)$ for its topological entropy. The problem of
	lowering entropy has a substantial history; see, for example,
	\cite{ShubWeiss1991,Lindenstrauss1995}. 
	Its classical form concerns factors rather than subsets. 
	In analogy with the measure-theoretic fact that positive entropy can often be reduced by passing to suitable factors, Shub and Weiss asked whether a topological dynamical system of positive entropy must admit factors of smaller prescribed entropy \cite{ShubWeiss1991}. 
	They proved this in the uniquely ergodic case, but also showed that the general topological problem is delicate by constructing an infinite-entropy system for which every nontrivial factor has infinite entropy. Lindenstrauss later gave an affirmative answer for finite-dimensional systems \cite{Lindenstrauss1995}. 
	In a subsequent development, using mean topological dimension and the small boundary property, he proved an affirmative result for extensions of nontrivial minimal systems with zero mean topological dimension \cite{Lindenstrauss1999}; see also \cite{LindenstraussWeiss2000} for the theory of mean topological dimension.
	
	A different but closely related viewpoint is to lower entropy not by passing to factors, but by passing to subsets. This perspective is partly motivated by dimension theory: for a Borel subset of Euclidean space, intermediate Hausdorff dimensions can be realized by suitable subsets \cite{Falconer1990,Mattila1995}. In topological dynamics, Ye and Zhang's theory of entropy points showed that the entropy of a compact set can be captured by countable compact subsets with controlled accumulation structure \cite{YeZhang2007}. This made it natural to ask whether, for every value between zero and the entropy of a system or a compact set, one can find a compact subset realizing exactly that value. It is important to note that this subset problem is genuinely different from the factor problem.
	Motivated by the problem of realizing intermediate entropy values over subsets, Huang, Ye, and Zhang introduced the notions of lowerability, hereditary lowerability and hereditary uniform lowerability in \cite{HYZ2010}. They proved, among other things, that every system with finite topological entropy is lowerable and characterized the stronger hereditary uniform lowerability by asymptotic $h$-expansiveness.
	
	In their subsequent paper \cite{HYZ2014}, the same authors proved that every  finite-entropy system is D-lowerable and they also constructed systems that are lowerable but not hereditarily lowerable. Their examples have infinite topological entropy, so the remaining finite-entropy problem was formulated as follows.

	\begin{question}[{\cite[p.~4430]{HYZ2014}, Question $2'$}]
		\label{ques:finite}
		Is there a topological dynamical system with finite entropy which is not
		hereditarily lowerable?
	\end{question}
	
	Immediately after this question, they proposed the following obstruction in
	the infinite-entropy setting.
	
	\begin{conjecture}[{\cite[p.~4430]{HYZ2014}}]
		\label{conj:infinite}
		If a topological dynamical system admits an ergodic invariant probability
		measure with infinite entropy, then it is not hereditarily lowerable.
	\end{conjecture}
	
	Our two main results answer \cref{ques:finite} and prove
	\cref{conj:infinite}.
	
	\begin{theorem}
		\label{thm:finite-main}
		If $\htop(T,X)<+\infty$, then $(X,T)$ is hereditarily lowerable.
	\end{theorem}
	
	\begin{theorem}
		\label{thm:infinite-main}
		Suppose that there is a measure $\mu\in \mathcal{M}^e(X,T)$ such that
		$h_\mu(T)=+\infty$. Then $(X,T)$ is not hereditarily lowerable. In fact, there
		is a compact set $K\subset X$ satisfying
		\[
		h(T,K)=+\infty
		\quad\text{and}\quad
		h(T,L)\in\{0,+\infty\}
		\]
		for every compact set $L\subset K$.
	\end{theorem}
	
	\Cref{thm:finite-main} shows that the answer to \cref{ques:finite} is negative.
	\Cref{thm:infinite-main} establishes exactly the assertion of
	\cref{conj:infinite}. These statements do not give a complete classification
	of systems with infinite topological entropy: such a system need not carry an
	ergodic measure with infinite entropy.
	
	The paper is organized as follows. 
	In \cref{sec:preliminaries}, we introduce the preliminary notions and tools used throughout the paper. In particular, we introduce the language of finite rooted trees and dynamical name trees, and establish the correspondence between tree structures and the growth of dynamical names. 
	In \cref{sec:finite-part} we prove
	\cref{thm:finite-main}. A finite prefix-tree lemma permits simultaneous control
	of name growth for a refining sequence of clopen partitions. This gives exact
	entropy lowering inside a countable compact set with a unique accumulation
	point in zero-dimensional systems. A local-entropy theorem and a zero-dimensional principal extension then
	complete the proof. In \cref{sec:infinite-obstruction} we prove \cref{thm:infinite-main}. We use 
	relative Sinai theorem to create an independent finite-valued uniform i.i.d. process, disintegrate
	over the resulting residual codes, and obtain one compact fibre on which every
	positive name entropy is amplified without bound.
	
	\section{Preliminaries}
	\label{sec:preliminaries}
	
	Throughout, all logarithms are to base $e$. We write $\mathcal{M}(X,T)$ for the invariant Borel
	probability measures and $\mathcal{M}^e(X,T)$ for the ergodic ones. Standard background
	on topological and measure-theoretic entropy may be found in
	\cite{Walters1982,Downarowicz2011}; background on measure disintegration and relative
	entropy may be found in \cite{Glasner2003,EinsiedlerWard2011}.
	
	\subsection{Topological entropy of compact sets}
	
	If $\cU$ is a finite open cover of $X$, put
	\[
	\cU_0^{n-1}:=\join_{i=0}^{n-1}T^{-i}\cU
	\]
	and let $N(\cU_0^{n-1},K)$ be the least number of members of
	$\cU_0^{n-1}$ needed to cover $K$. Define
	\[
	h(T,\cU,K)
	:=
	\limsup_{n\to\infty}
	\frac1n\log N(\cU_0^{n-1},K),
	\quad
	h(T,K):=\sup_{\cU}h(T,\cU,K).
	\]
	For compact $K$, this agrees with the usual separated- and spanning-set
	description. More precisely, for a compatible metric $d$, set
	\[
	d_n(x,y):=\max_{0\leq i<n}d(T^ix,T^iy).
	\]
	If $r_n(d,T,\varepsilon,K)$ and $s_n(d,T,\varepsilon,K)$ denote the least
	spanning and greatest separated cardinalities, respectively, then their
	small-scale exponential growth gives $h(T,K)$, i.e. 
	\[h(T,K)=\lim_{\varepsilon\to 0}\limsup_{n\to\infty}\frac{1}{n}\log r_n(d,T,\varepsilon,K)=\lim_{\varepsilon\to 0}\limsup_{n\to\infty}\frac{1}{n}\log s_n(d,T,\varepsilon,K);\] see
	\cite{Bowen1973,Misiurewicz2004}.
	
	\begin{definition}
		\label{def:lowerable}
		A topological dynamical system $(X,T)$ is
		\begin{enumerate}
			\item \emph{lowerable} if, for every $0\leq h\leq h(T,X)$, there is a
			nonempty compact set $K_h\subset X$ with $h(T,K_h)=h$;
			\item \emph{hereditarily lowerable} if every nonempty compact set
			$K\subset X$ is lowerable, that is, if for every
			$0\leq h\leq h(T,K)$ there is a nonempty compact set
			$K_h\subset K$ with $h(T,K_h)=h$.
		\end{enumerate}
	\end{definition}
	
	\subsection{Names and partitions}
	
	For a finite Borel partition $\alpha$, write
	\[
	\alpha_0^{n-1}:=\join_{i=0}^{n-1}T^{-i}\alpha,
	\quad
	\alpha^T:=\join_{i\in\Z}T^{-i}\alpha.
	\]
	For a set $E\subset X$, let
	\[
	\Nname_\alpha(n,E)
	:=
	\#\{A\in\alpha_0^{n-1}:A\cap E\neq\varnothing\}
	\]
	for $n\geq1$; for a nonempty set $E$ we use the harmless convention
	$\Nname_\alpha(0,E)=1$, corresponding to the empty name, which we denote by $\emptyset$.  Define its upper
	name-growth rate by
	\[
	\overline h_\alpha(E)
	:=
	\limsup_{n\to\infty}\frac1n\log\Nname_\alpha(n,E).
	\]
	When $\alpha$ is clopen we also write this quantity as
	$h_\alpha(T,E)$. The symbol $\overline h_\alpha(E)$ is a combinatorial
	name-growth rate and should not be confused with the Kolmogorov--Sinai entropy
	of an invariant measure. We write $\alpha\prec\beta$ when $\beta$ refines
	$\alpha$, and set
	\[
	\mesh(\alpha):=\max\{\diam(A):A\in\alpha\}.
	\]
	As is well known. every compact zero-dimensional metric space admits a refining sequence of
	finite clopen partitions whose meshes tend to zero.

	\subsection{Finite rooted trees and dynamical name trees}\label{sec:name trees}

	We record the tree terminology used below.  A \emph{finite rooted tree} is a
	finite connected graph without cycles, together with a distinguished vertex
	$o$, called the \emph{root}.  Every vertex $v\neq o$ has a unique
	\emph{parent}, namely the adjacent vertex on the path from $v$ to $o$; the
	other adjacent vertices farther from $o$ are its \emph{children}.  A vertex
	$w$ is a \emph{descendant} of $v$ if the path from $w$ to $o$ passes through
	$v$.  The \emph{level} of $v$ is its graph distance from $o$, and the
	\emph{depth} or \emph{height} of the tree is the largest level of any vertex.  A vertex with
	no children is a \emph{leaf}.

	In what follows, only trees whose designated leaves lie at one fixed terminal level will be used.  Thus, if the depth is $n$, the set of terminal leaves is denoted by
	$\mathcal L_n$ and its cardinality by
	\[
	A:=|\mathcal L_n|.
	\]
	For $0\leq\ell\leq n$, a vertex at level $\ell$ is called \emph{active} if
	it has a descendant in $\mathcal L_n$.  The set of active vertices at level
	$\ell$ is denoted by $\mathcal V_\ell$, and
	\[
	B_\ell:=|\mathcal V_\ell|.
	\]
	For $v\in\mathcal V_\ell$, let
	\[
	A_v:=\#\{w\in\mathcal L_n:w\text{ is a descendant of }v\}.
	\]
	The descendant sets belonging to distinct vertices of the same level are
	disjoint and partition $\mathcal L_n$; in particular,
	\begin{equation}
		\sum_{v\in\mathcal V_\ell}A_v=A.
		\label{eq:descendant-partition}
	\end{equation}
	If $\mathcal F\subset\mathcal L_n$, a vertex $v\in\mathcal V_\ell$ is said to be
	\emph{occupied by $\mathcal F$} when at least one leaf of $\mathcal F$ is a
	descendant of $v$.  We write
	\[
	b_\ell(\mathcal F)
	:=
	\#\{v\in\mathcal V_\ell:v\text{ is occupied by }\mathcal F\}.
	\]

	After fixing an order on the children of every vertex, the \emph{depth-first
	search order} (DFS order) is obtained by exploring completely the subtree of
	one child before passing to the next child.  In this order the descendant
	leaves of every vertex form a consecutive interval; see, for example,
	\cite{Tarjan1972,CLRS2022}.  If the $A$ leaves are
	listed as $0,1,\ldots,A-1$ and the list is read modulo $A$, then
	\[
	\{u,u+1,\ldots,u+M-1\}\subset\Z/A\Z
	\]
	is called a \emph{cyclic block of $M$ leaves}.

	We next specialize this terminology to dynamical names.  Let $\alpha$ be a
	finite partition and let $P\subset X$ be finite.  For $x\in P$, its
	$\alpha$-name of length $n$ is the word
	\[
	\operatorname{name}_{\alpha,n}(x)
	=
	\bigl(\alpha(x),\alpha(Tx),\ldots,\alpha(T^{n-1}x)\bigr),
	\]
	where $\alpha(y)$ denotes the atom of $\alpha$ containing $y$.  Suppose that
	the points of $P$ have pairwise distinct $\alpha$-names of length $n$.  The
	\emph{$\alpha$-name tree of $P$ at depth $n$} has the empty word as its root;
	its active vertices at level $\ell$ are the distinct prefixes of length
	$\ell$ of these names, and an edge joins each nonempty prefix to the prefix
	obtained by deleting its last symbol.  The full length-$n$ names are its
	terminal leaves.

	\begin{remark}[Dictionary between rooted trees and dynamical names]
		\label{rem:name-tree-dictionary}
		For the $\alpha$-name tree of $P$ described above, the correspondence needed
		later is
		\begin{align*}
		\text{terminal leaf}
		&\longleftrightarrow
		\text{a distinct length-$n$ $\alpha$-name represented in $P$},\\
		\text{active vertex at level $\ell$}
		&\longleftrightarrow
		\text{a distinct length-$\ell$ prefix represented in $P$},\\
		B_\ell
		&=
		\Nname_\alpha(\ell,P),\\
		b_\ell(\mathcal F)
		&=
		\Nname_\alpha(\ell,F).
		\end{align*}
		Here $F\subset P$ is the set of points whose terminal names form
		$\mathcal F$.  Thus selecting leaves in the finite rooted tree is exactly
		selecting points of $P$, while controlling occupied vertices simultaneously
		at all levels is exactly controlling the numbers of dynamical name prefixes
		at all corresponding lengths.  The identity
		\eqref{eq:descendant-partition} says that every terminal name has a unique
		prefix of each prescribed length.
	\end{remark}
	

	\section{Hereditary lowerability of finite-entropy systems}
	\label{sec:finite-part}
    In this section, we prove that finite-entropy systems are hereditarily lowerable. We first work in the zero-dimensional setting, where entropy of compact sets can be computed through name growth with respect to refining clopen partitions. A finite rooted tree thinning argument is then used to realize prescribed intermediate name-growth rates inside countable compact sets with a unique accumulation point. Combining this construction with properties of relative entropy function gives hereditary lowerability for zero-dimensional finite-entropy systems, and a zero-dimensional principal extension transfers the conclusion to general finite-entropy systems.
    
	\subsection{Name entropy in zero-dimensional systems}
	\label{sec:name-entropy}
	
	\begin{lemma}
		\label{lem:clopen-detects}
		Let $X$ be zero-dimensional, and let
		\[
		\alpha_1\prec\alpha_2\prec\cdots
		\]
		be refining finite clopen partitions such that
		$\mesh(\alpha_m)\to0$. Then, for every compact set $E\subset X$,
		\[
		h(T,E)=\sup_{m\geq1}h_{\alpha_m}(T,E).
		\]
	\end{lemma}
	
	\begin{proof}
		Each $\alpha_m$ is a finite open cover, so
		$h_{\alpha_m}(T,E)\leq h(T,E)$.  Conversely, let $\cU$ be a finite open
		cover.  The Lebesgue number lemma states that there is a number $\delta>0$
		such that every subset of $X$ having diameter less than $\delta$ is contained
		in some member of $\cU$.  Choose $m$ with $\mesh(\alpha_m)<\delta$.  Then,
		for every atom $A\in\alpha_m$, there exists $U_A\in\cU$ such that
		$A\subset U_A$; in other words, $\alpha_m\succ\cU$.  It follows that
		\[
			N(\cU_0^{n-1},E)
			\leq
			\Nname_{\alpha_m}(n,E),
		\]
		and consequently
		\[
		h(T,\cU,E)\leq h_{\alpha_m}(T,E).
		\]
		Taking the supremum over $\cU$ proves the result.
	\end{proof}
	
	\begin{lemma}
		\label{lem:finite-union}
		For a finite clopen partition $\alpha$ and sets $E_1,\ldots,E_r\subset X$,
		\[
		h_\alpha\left(T,\bigcup_{i=1}^rE_i\right)
		=
		\max_{1\leq i\leq r}h_\alpha(T,E_i).
		\]
		Consequently, if $F\subset E$ is finite and $E\setminus F\neq\varnothing$,
		then
		\[
		h_\alpha(T,E\setminus F)=h_\alpha(T,E).
		\]
	\end{lemma}
	
	\begin{proof}
		For $1\leq i\leq r$, put
		$N_i(n):=\Nname_\alpha(n,E_i)$ and
		$M(n):=\max_{1\leq i\leq r}N_i(n)$.  Monotonicity and subadditivity of the
		name count under unions give
		\[
		M(n)
		\leq
		\Nname_\alpha\left(n,\bigcup_{i=1}^rE_i\right)
		\leq
		\sum_{i=1}^rN_i(n)
		\leq
		rM(n).
		\]
		Since $(\log r)/n\to0$, the claimed finite union formula follows.

		If $F$ is finite, then $\Nname_\alpha(n,F)\leq|F|$ for every $n$, so
		$h_\alpha(T,F)=0$.  Applying the finite-union formula to
		$E=(E\setminus F)\cup F$ gives the finite deletion assertion.
	\end{proof}
	
	\subsection{DFS simultaneous leaf thinning}
	\label{sec:tree}
	
	The proof of the next lemma is an application of the elementary probabilistic
	method; compare \cite{AlonSpencer2016}.
	
	\begin{lemma}\label{lem:tree}
		Let $\mathcal T$ be a finite rooted tree of depth $n$, with exactly $A$
		terminal leaves at level $n$. For $1\leq\ell\leq n$, let $B_\ell$ denote
		the number of active vertices at level $\ell$. Then, for every integer $M$
		satisfying $1\leq M\leq A$, there exists a set $\mathcal F$ consisting of
		exactly $M$ terminal leaves such that
		\begin{equation}
			b_\ell
			\leq
			2\ell^2\left(1+\frac{MB_\ell}{A}\right),
			\quad 1\leq\ell\leq n.
			\label{eq:tree-bound}
		\end{equation}
		where $b_\ell=b_\ell(\mathcal F)$ denotes the number of active vertices at
		level $\ell$ occupied by $\mathcal F$.
	\end{lemma}
	
	\begin{proof}
		Fix an order on the children of every vertex and list the terminal leaves in
		the induced DFS order.  Identify their positions with
		$\{0,1,\ldots,A-1\}$ and, when cyclic blocks are considered, with
		$\Z/A\Z$.  For an active vertex $v$, let $I_v$ be the set of positions of
		its descendant terminal leaves.  The defining property of DFS implies that
		$I_v$ is an interval in the linear leaf order, and its cardinality is $A_v$.

		Choose a random variable $U$ uniformly distributed on $\Z/A\Z$ and let
		\[
			\mathcal F_U
			:=
			\{U,U+1,\ldots,U+M-1\}\pmod A.
		\]
		This is a set of exactly $M$ terminal leaves.  For an active vertex $v$,
		write $X_v(U)$ for the indicator of the event that $v$ is occupied by
		$\mathcal F_U$.  Equivalently,
		\[
			X_v(U)=1
			\quad\Longleftrightarrow\quad
			\mathcal F_U\cap I_v\neq\varnothing.
		\]
		If $I_v$ has $A_v$ elements, the possible starting positions $U$ for which
		the cyclic block meets $I_v$ are contained in the cyclic enlargement of
		$I_v$ by the preceding $M-1$ positions.  This enlargement has at most
		$A_v+M-1$ elements.  Consequently,
		\begin{equation}
			\mathbb P\bigl(X_v=1\bigr)
			\leq
			\min\left\{1,\frac{A_v+M-1}{A}\right\}
			\leq
			\frac{A_v+M-1}{A}.
			\label{eq:vertex-hit-probability}
		\end{equation}

		At level $\ell$ the random number of occupied vertices is
		\[
			b_\ell(\mathcal F_U)
			=
			\sum_{v\in\mathcal V_\ell}X_v(U).
		\]
		By linearity of expectation, \eqref{eq:vertex-hit-probability}, and
		\eqref{eq:descendant-partition},
		\begin{align}
			\mathbb E b_\ell(\mathcal F_U)
			&\leq
			\sum_{v\in\mathcal V_\ell}
			\frac{A_v+M-1}{A}\notag\\
			&=
			\frac{1}{A}\sum_{v\in\mathcal V_\ell}A_v
			+\frac{(M-1)B_\ell}{A}\notag\\
			&=
			1+\frac{(M-1)B_\ell}{A}
			\leq
			1+\frac{MB_\ell}{A}.
			\label{eq:expected-prefix}
		\end{align}
		This estimate holds separately at every level.  To obtain one leaf set for
		which all the level estimates hold simultaneously, define the nonnegative
		random variable
		\[
			Z(U)
			:=
			\sum_{\ell=1}^n
			\frac{b_\ell(\mathcal F_U)}
			{\ell^2(1+MB_\ell/A)}.
		\]
		Using \eqref{eq:expected-prefix} for each term gives
		\[
			\mathbb E Z
			\leq
			\sum_{\ell=1}^n\frac1{\ell^2}
			<
			\sum_{\ell=1}^{\infty}\frac1{\ell^2}
			=
			\frac{\pi^2}{6}<2.
		\]
		It follows that there is a starting position $u\in\Z/A\Z$ for which
		$Z(u)<2$.  Put $\mathcal F:=\mathcal F_u$. then for each $1\leq\ell\leq n$ one has
		\[
			\frac{b_\ell(\mathcal F)}
			{\ell^2(1+MB_\ell/A)}
			\leq Z(u)<2,
		\]
		which yields \eqref{eq:tree-bound} simultaneously for
		all levels $\ell$.
	\end{proof}
	
	\begin{remark}
		The loss in \eqref{eq:tree-bound} is polynomial in the current level $\ell$,
		not merely in the total depth $n$. This distinction is essential when
		$\ell\ll n$.
	\end{remark}
	
	\subsection{Lowerability of countable compact sets with a unique accumulation point}
	\label{sec:one-limit}
	
	\begin{lemma}
		\label{lem:one-limit-point}
		Let $(X,T)$ be a zero-dimensional topological dynamical system and let
		\[
		C=\{x_*\}\cup\{x_j:j\geq1\},
		\quad x_j\longrightarrow x_*,
		\]
		be a countable compact set whose unique accumulation point is $x_*$. If
		\[
		H:=h(T,C)<+\infty,
		\]
		then for every $0<h<H$ there is a compact set $F\subset C$ such that
		$h(T,F)=h$. Moreover, $F$ has at most the one accumulation point $x_*$.
	\end{lemma}
	
	\begin{proof}
		Choose refining clopen partitions $\{\alpha_m\}_{m\geq1}$ as in
		\cref{lem:clopen-detects}, and put
		\[
		a_m:=h_{\alpha_m}(T,C).
		\]
		Since $\alpha_{m+1}\succ \alpha_m$, the sequence $(a_m)$ is
		nondecreasing.  Moreover, \cref{lem:clopen-detects} gives
		\begin{equation}
			a_m\uparrow H.
			\label{eq:am-to-H}
		\end{equation}
		Choose $m_0$ such that $a_m>0$ for every $m\geq m_0$.  Let $m(s)$ be a
		sequence in which every $m\geq m_0$ occurs infinitely often, such as, $m_0;~m_0,~ m_0+1;~ m_0,~ m_0+1,~ m_0+2; \dots$.  Recursively
		choose a strictly increasing sequence of integers $q_s$ satisfying
		\begin{equation}
			q_s\geq\max\{m(s),s\}.
			\label{eq:qs-choice}
		\end{equation}
		In particular, $q_s\to\infty$, and $\alpha_{q_s}\succ \alpha_{m(s)}$.  Choose $\delta_s\downarrow0$ so that
		\[
		c_s:=a_{m(s)}-\delta_s>0.
		\]
		Set
		\begin{equation}
			D_s:=H+\delta_s,
			\quad
			r_s:=\frac{h\,c_s}{D_s}.
			\label{eq:csr}
		\end{equation}
		Since $0<c_s<H<D_s$ and $0<h<D_s$, we have
		\begin{equation}
			0<r_s<c_s,
			\quad 0<r_s<h.
			\label{eq:rs-bounds}
		\end{equation}
		
		We now construct integers
		\[
		0=n_0<L_1<n_1<L_2<n_2<\cdots
		\]
		and finite subsets $F_s$.  Suppose that $n_{s-1}$ has already been
		chosen.  Since $h_{\alpha_{q_s}}(T,C)\leq H<D_s$, this permits us to choose $L_s>n_{s-1}$ such that
		\begin{equation}
			\Nname_{\alpha_{q_s}}(\ell,C)
			\leq e^{D_s\ell},~ \forall\ell\geq L_s.
			\label{eq:fine-upper}
		\end{equation}
		Let $Q_s$ be the atom of $(\alpha_{q_s})_0^{L_s-1}$ containing $x_*$.  It
		is a clopen neighbourhood of $x_*$.  Since $x_j\to x_*$, we may choose
		strictly increasing tail indices $J_s$ so that
		\[
		C_s:=\{x_*\}\cup\{x_j:j\geq J_s\}
		\]
		lies in $Q_s$.  Equivalently, all points of $C_s$ have the same
		$\alpha_{q_s}$-name of length $\ell$, $1\leq \ell\leq L_s$.  Since $C\setminus C_s$ is
		finite, \cref{lem:finite-union} gives
		\begin{equation}
			h_{\alpha_{m(s)}}(T,C_s)=a_{m(s)}.
			\label{eq:tail-entropy}
		\end{equation}
		The number $c_s$ is strictly smaller than the entropy in
		\eqref{eq:tail-entropy}.  Hence there are arbitrarily large integers $n$
		for which $\Nname_{\alpha_{m(s)}}(n,C_s)\geq e^{c_sn}$.  Choose such an
		$n_s>L_s$, increasing it if necessary so that $e^{r_sn_s}\geq2$.  Then
		\begin{equation}
			A_s:=\Nname_{\alpha_{m(s)}}(n_s,C_s)
			\geq e^{c_sn_s}
			\label{eq:coarse-lower}
		\end{equation}
		
		There are exactly $A_s$ atoms of $(\alpha_{m(s)})_0^{n_s-1}$ meeting
		$C_s$.  Select one representative from each of them.  In the atom containing
		$x_*$ choose a representative different from $x_*$. 
	    Denote the resulting
		set by $P_s\subset C_s\setminus\{x_*\}$.  Then $|P_s|=A_s$, and its points
		have pairwise distinct $\alpha_{m(s)}$-names of length $n_s$.
		
		Because $\alpha_{q_s}\succ\alpha_{m(s)}$, a fine name determines its
		coarse name.  Thus the points of $P_s$ also have pairwise distinct
		$\alpha_{q_s}$-names of length $n_s$.  Form the $\alpha_{q_s}$-name tree
		$\mathcal T_s$ of $P_s$ at depth $n_s$ as in
		\cref{sec:name trees}.  It has $A_s$ terminal leaves.  For
		$0\leq\ell\leq n_s$, let
		\[
			B_{s,\ell}
			:=
			\Nname_{\alpha_{q_s}}(\ell,P_s)
		\]
		be the number of active vertices at level $\ell$.  Since
		$P_s\subset C_s\subset Q_s$, all points of $P_s$ have one common prefix of
		length at most $L_s$.  For longer prefixes, use $P_s\subset C$ and
		\eqref{eq:fine-upper}.  Hence
		\begin{equation}
			B_{s,\ell}=
			\begin{cases}
				1, & 0\leq\ell\leq L_s,\\
				\leq e^{D_s\ell}, & L_s<\ell\leq n_s.
			\end{cases}
			\label{eq:prefix-count}
		\end{equation}
		Put
		\begin{equation}
			M_s:=\left\lfloor e^{r_sn_s}\right\rfloor.
			\label{eq:Ms}
		\end{equation}
		The conditions $e^{r_sn_s}\geq2$, $r_s<c_s$, and
		\eqref{eq:coarse-lower} imply
		\[
			1\leq M_s\leq e^{r_sn_s}<e^{c_sn_s}\leq A_s.
		\]
		Apply \cref{lem:tree} to $\mathcal T_s$ with
		$A=A_s$, $M=M_s$, and $B_\ell=B_{s,\ell}$.  It selects $M_s$ terminal
		leaves and therefore, by \cref{rem:name-tree-dictionary}, a point set
		$F_s\subset P_s$ with $|F_s|=M_s$.  Define
		\[
			b_{s,\ell}
			:=
			\Nname_{\alpha_{q_s}}(\ell,F_s).
		\]
		This is precisely the number of level-$\ell$ vertices occupied by the
		selected leaves.  The \cref{lem:tree} gives, simultaneously for
		$1\leq\ell\leq n_s$,
		\begin{equation}
			b_{s,\ell}
			\leq
			2\ell^2\left(1+\frac{M_sB_{s,\ell}}{A_s}\right).
			\label{eq:stage-tree-bound}
		\end{equation}
		
		Consider first a length $\ell$ satisfying
		\[
		L_s\leq\ell\leq\frac{c_sn_s}{D_s},
		\]
		Equations \eqref{eq:coarse-lower}, \eqref{eq:prefix-count}, and
		\eqref{eq:Ms} imply
		\[
		\frac{M_sB_{s,\ell}}{A_s}
		\leq
		\exp(r_sn_s+D_s\ell-c_sn_s).
		\]
		The definition of $r_s$ was chosen so that
		\[
		\begin{aligned}
		&r_sn_s+D_s\ell-c_sn_s-h\ell\\
		&\quad=
		(D_s-h)\left(\ell-\frac{c_sn_s}{D_s}\right)\leq0.
		\end{aligned}
		\]
		Hence $M_sB_{s,\ell}/A_s\leq e^{h\ell}$.  Since
		$1\leq e^{h\ell}$, \eqref{eq:stage-tree-bound} yields
		\begin{equation}
			b_{s,\ell}\leq4\ell^2e^{h\ell}.
			\label{eq:active-bound-one}
		\end{equation}
		If instead $c_sn_s/D_s<\ell\leq n_s$, the number of names is bounded by
		the number of selected points, and therefore
		\begin{equation}
			b_{s,\ell}\leq M_s
			\leq e^{r_sn_s}
			<e^{h\ell}.
			\label{eq:active-bound-two}
		\end{equation}
		For $1\leq\ell\leq L_s$ all selected points share one prefix.  For
		$\ell>n_s$ there are at most $M_s$ names, and
		$M_s<e^{h\ell}$ by \eqref{eq:rs-bounds}.  Combining these observations with
		\eqref{eq:active-bound-one} and \eqref{eq:active-bound-two}, we obtain
		\begin{equation}
			\Nname_{\alpha_{q_s}}(\ell,F_s)
			\leq4\ell^2e^{h\ell},~\forall\ell\geq1.
			\label{eq:cloud-upper}
		\end{equation}
		On the other hand, $F_s\subset P_s$ and the points of $P_s$ have distinct
		$\alpha_{m(s)}$-names of length $n_s$, so
		\begin{equation}
			\Nname_{\alpha_{m(s)}}(n_s,F_s)=M_s.
			\label{eq:cloud-lower}
		\end{equation}
		
		Define
		\begin{equation}
			F:=\{x_*\}\cup\bigcup_{s\geq1}F_s.
			\label{eq:F-definition}
		\end{equation}
		Every $F_s$ is finite and is contained in $\{x_j:j\geq J_s\}$, while
		$J_s\to\infty$.  
		It follows that $F$ is closed in the compact space $C$ and $x_*$ is its only possible
		accumulation point.
		
		We first establish the entropy upper bound.  Fix $k\geq1$.  By
		\eqref{eq:qs-choice}, $q_s\geq s$; hence $q_s\geq k$ whenever $s\geq k$.
		For every sufficiently large $n$, there is a unique $s=s(n)$ such that
		\begin{equation}
			n_{s-1}<n\leq n_s.
			\label{eq:stage-containing-n}
		\end{equation}
		As $n\to\infty$, one has $s(n)\to\infty$.  After discarding finitely many
		values of $n$, which does not affect an upper exponential growth rate, we may
		therefore assume that $s\geq k$ and $q_s\geq k$.

		Decompose $F$ into the earlier part $\bigcup_{t<s}F_t$, the current part $F_s$, and the later part $\bigcup_{t>s}F_t$.  
		By
		\eqref{eq:Ms}, \eqref{eq:rs-bounds}, and the strict increase of 
		$n_t$,
		\[
		\Nname_{\alpha_k}\left(n,\bigcup_{t<s}F_t\right)
		\leq\sum_{t<s}|F_t|
		\leq
		\sum_{t<s}e^{hn_t}
		\leq
		\sum_{j=1}^{n-1}e^{hj}
		\leq C_he^{hn},
		\]
		where $C_h$ depends only on $h$.  For the current part, the refinement
		$\alpha_k\prec\alpha_{q_s}$ means that every fine name determines a unique
		coarse name.  Hence \eqref{eq:cloud-upper} gives
		\begin{equation}
			\Nname_{\alpha_k}(n,F_s)
			\leq
			\Nname_{\alpha_{q_s}}(n,F_s)
			\leq4n^2e^{hn}.
			\label{eq:current-cloud-upper}
		\end{equation}
		Finally, if $t>s$, then
		\[
			n<L_t\quad\text{and}\quad q_t\geq t>s\geq k.
		\]
		By construction, $F_t\cup\{x_*\}$ lies in one atom of
		$(\alpha_{q_t})_0^{L_t-1}$.  Since $\alpha_{q_t} \succ \alpha_k$ and
		$n<L_t$, every point in each $F_t$ has the same
		$\alpha_k$-name of length $n$ as $x_*$.  All $F_t$, $t>s$  together with $x_*$ therefore
		contribute one name.  Summing these three contributions gives
		\begin{equation}
			\Nname_{\alpha_k}(n,F)
			\leq
			1+C_he^{hn}+4n^2e^{hn},
			\label{eq:global-upper}
		\end{equation}
		and therefore
		$h_{\alpha_k}(T,F)\leq h$.  This holds for every $k$, so
		\cref{lem:clopen-detects} gives
		\begin{equation}
			h(T,F)\leq h.
			\label{eq:entropy-upper}
		\end{equation}
		
		For the reverse inequality, fix $m\geq m_0$.  The schedule $m(s)$ assumes
		the value $m$ infinitely often.  Since $F_s\subset F$, \eqref{eq:cloud-lower} implies
		\[
		h_{\alpha_m}(T,F)
		\geq
		\limsup_{\substack{s\to\infty\\m(s)=m}}
		\frac1{n_s}\log M_s.
		\]
		Along this subsequence, $a_{m(s)}=a_m$ and $\delta_s\to0$, whence
		\[
			r_s
			=
			\frac{h(a_m-\delta_s)}{H+\delta_s}
			\longrightarrow
			\frac{h\,a_m}{H}.
		\]
		Since $n_s\to\infty$, it follows that
		\begin{equation}
			h_{\alpha_m}(T,F)
			\geq
			\frac{h\,a_m}{H}.
			\label{eq:fixed-m-lower}
		\end{equation}
		Taking the supremum over $m$ in \eqref{eq:fixed-m-lower} and using
		\eqref{eq:am-to-H}, we obtain
		\[
		h(T,F)
		\geq
		\sup_{m\geq m_0}\frac{h\,a_m}{H}
		=h.
		\]
		Together with \eqref{eq:entropy-upper}, this proves $h(T,F)=h$.
	\end{proof}
	
	\subsection{Hereditary lowerability of finite-entropy systems}
	\label{sec:finite-systems}
	We use the following local entropy facts. They are stated in
	\cite[Theorem~5.1]{HYZ2010}, based on the entropy point theory of \cite{YeZhang2007}. 
	
	We first recall the entropy function introduced in
	\cite[Definition~4.1]{YeZhang2007}. Let $d$ be a compatible metric on $X$.
	For $\varepsilon>0$ and a nonempty set $E\subset X$, put
	\[
	r(d,T,\varepsilon,E)
	:=
	\limsup_{n\to\infty}
	\frac{1}{n}
	\log r_n(d,T,\varepsilon,E),
	\]
	\begin{definition}[Entropy function]
		\label{def:entropy-function}
		Let $(X,T)$ be a topological dynamical system. For $x\in X$ and
		$\varepsilon>0$, define
		\[
		h_d(x,\varepsilon)
		:=
		\inf
		\left\{
		r(d,T,\varepsilon,C):
		C\text{ is a closed neighbourhood of }x\text{ in }X
		\right\}.
		\]
		Here a closed neighbourhood of $x$ means a closed set $C\subset X$
		such that $x\in\operatorname{int}_X C$.
		
		Since $h_d(x,\varepsilon)$ is nondecreasing as
		$\varepsilon\downarrow0$, the limit
		\[
		h_d(x)
		:=
		\lim_{\varepsilon\downarrow0}h_d(x,\varepsilon)
		\]
		exists in $[0,+\infty]$. Its value is independent of the choice of the
		compatible metric $d$. Thus the function
		\[
		h_X\colon X\longrightarrow[0,h(T,X)],
		\quad
		x\longmapsto h_X(x):=h_d(x),
		\]
		is called the entropy function of $(X,T)$. When there is no ambiguity,
		we write $h(x)$ instead of $h_X(x)$.
	\end{definition}
	
	For the argument below, we require the relative version of the entropy
	function given in \cite[Remark~5.13]{YeZhang2007}. Let $K\subset X$ be a
	nonempty compact set. For $x\in K$ and $\varepsilon>0$, define
	\[
	h_d(x,\varepsilon;K)
	:=
	\inf
	\left\{
	r(d,T,\varepsilon,C):
	C\text{ is a closed neighbourhood of }x\text{ in }K
	\right\},
	\]
	where ``closed neighbourhood in $K$'' refers to the relative topology of
	$K$; equivalently, $C$ is closed in $K$ and
	$x\in\operatorname{int}_K C$. Define
	\begin{equation}
		h(x,K)
		:=
		\lim_{\varepsilon\downarrow0}
		h_d(x,\varepsilon;K).
		\label{eq:relative-entropy-function}
	\end{equation}
	As in the case $K=X$, this value is independent of the compatible metric
	$d$. Hence
	\[
	h_K\colon K\longrightarrow[0,h(T,K)],
	\quad
	h_K(x):=h(x,K),
	\]
	is a well-defined function, called the relative entropy function of $K$.
	In particular,
	\[
	h_X(x)=h(x,X).
	\]
	
	The properties of the relative entropy function needed below follow from
	\cite[Remark~5.13]{YeZhang2007}; see also
	\cite[Theorem~5.1]{HYZ2010}.
	
	\begin{proposition}[Relative entropy function]
		\label{prop:local-entropy}
		Let $K$ be a nonempty compact subset of a topological dynamical system
		$(X,T)$, and let $h_K$ be its relative entropy function. Then:
		\begin{enumerate}
			\item
			$\sup_{x\in K}h_K(x)
			=
			\sup_{x\in K}h(x,K)
			=
			h(T,K).$
			
			\item For every $x\in K$ satisfying $h_K(x)>0$, there exists a
			countable compact set $C_x\subset K$ whose unique accumulation point
			in $X$ is $x$ and such that
			\[
			h(T,C_x)=h_K(x)=h(x,K).
			\]
		\end{enumerate}
	\end{proposition}
	
	\begin{theorem}[Zero-dimensional finite-entropy case]
		\label{thm:zero-dimensional}
		If $(X,T)$ is zero-dimensional and $\htop(T,X)<+\infty$, then $(X,T)$ is
		hereditarily lowerable.
	\end{theorem}
	
	\begin{proof}
		Let $K\subset X$ be nonempty and compact, and let
		$0\leq h\leq h(T,K)$. If $h=0$, take a singleton; if $h=h(T,K)$, take $K$
		itself. Suppose that $0<h<h(T,K)$. By
		\cref{prop:local-entropy} (i), there is an $x\in K$ such that $h(x,K)>h$.
		\cref{prop:local-entropy} (ii) then gives a countable compact set $C_x\subset K$ with unique
		accumulation point $x$ and
		\[
		h(T,C_x)=h(x,K)>h.
		\]
		Since $h(T,C_x)\leq\htop(T,X)<+\infty$,
		\cref{lem:one-limit-point} supplies a compact set
		$K_h\subset C_x\subset K$ satisfying $h(T,K_h)=h$.
	\end{proof}
	
	\begin{remark}
		The countable compact subset having the same entropy as $K$ provided by
		\cite[Theorem~5.1(3)]{HYZ2010}need not, in general, have a unique
		accumulation point.  For a strict intermediate value
		$0<h<h(T,K)$, it is enough to find a point whose local entropy is strictly
		larger than $h$ and then use \cref{prop:local-entropy} (ii). The endpoint
		$h=h(T,K)$ is realized by $K$ itself.
	\end{remark}
	
	We now pass from zero-dimensional systems to arbitrary systems. We use two
	standard facts about principal extensions. Every finite-entropy system admits
	a zero-dimensional principal extension; see
	\cite{BoyleDownarowicz2004,DownarowiczHuczek2013} and the proof
	of \cite[Theorem~3.8]{HYZ2014}. If
	\[
	\pi\colon(Y,S)\longrightarrow(X,T)
	\]
	is such an extension and $(X,T)$ has finite entropy, then the relative
	topological entropy equals zero. The conditional variational principles in
	\cite{LedrappierWalters1977,DownarowiczSerafin2002} and the compact-set
	inequality \cite[Proposition~7.3]{HYZ2010} consequently give
	\begin{equation}
		h(S,A)=h(T,\pi(A))
		\quad\text{for every compact }A\subset Y.
		\label{eq:compact-entropy-preserved}
	\end{equation}
	
	\begin{proof}[Proof of \cref{thm:finite-main}]
		Assume $\htop(T,X)<+\infty$, and choose a zero-dimensional principal
		extension
		\[
		\pi\colon(Y,S)\longrightarrow(X,T).
		\]
		By \eqref{eq:compact-entropy-preserved},
		\[
		\htop(S,Y)=\htop(T,X)<+\infty.
		\]
		Hence $(Y,S)$ is hereditarily lowerable by
		\cref{thm:zero-dimensional}.  Finally, since $\pi$ is a principal
		extension and $\htop(T,X)<+\infty$,
		\cite[Proposition~7.8 (2)]{HYZ2010} implies that $(X,T)$ is
		hereditarily lowerable.
	\end{proof}
	
	\begin{corollary}
		\label{cor:question}
		There is no finite-topological-entropy system which fails to be hereditarily
		lowerable. Consequently, the answer to \cref{ques:finite} is negative.
	\end{corollary}
	
	
	\section{The infinite-measure-entropy obstruction}
	\label{sec:infinite-obstruction}
    In this section, we establish that infinite-measure-entropy provides a fundamental obstruction to hereditary lowerability. More precisely, we show that in systems with infinite-measure-entropy, there exists a compact subset whose entropy is infinite, and such that no intermediate entropy values can be realized on its compact subsets. The argument begins with quantitative estimates comparing the growth of partition names with the  entropy of the whole system. We then apply a relative version of Sinai’s theorem to construct independent uniform i.i.d. processes associated with increasingly refined partitions. By disintegrating the original measure, which has infinite entropy, over the resulting residual coding part, we obtain a compact fibre that carries infinite entropy. Finally, we show that the structure of this residual coding part enforces a dichotomy: every compact subset of this fibre has either zero entropy or infinite entropy, thereby excluding hereditary lowerability in this setting.
    
	\subsection{Counting lemmas}
	\label{sec:counting}
	
	\begin{lemma}
		\label{lem:subsequence}
		For every fixed integer $q\geq1$ and every $L\subset X$,
		\[
		\limsup_{n\to\infty}
		\frac1{qn}\log\Nname_\alpha(qn,L)
		=
		\overline h_\alpha(L).
		\]
	\end{lemma}
	
	\begin{proof}
		The function $n\mapsto\Nname_\alpha(n,L)$ is nondecreasing. Given $r$, put
		$n=\lceil r/q\rceil$. Then $r\leq qn<r+q$, so
		\[
		\frac1r\log\Nname_\alpha(r,L)
		\leq
		\frac{qn}{r}\,
		\frac1{qn}\log\Nname_\alpha(qn,L).
		\]
		Taking the limsup gives one inequality; the other follows because the
		right-hand side is computed along a subsequence of the original sequence.
	\end{proof}
	
	\begin{lemma}
		\label{lem:mesh}
		Let $\{\alpha_m\}$ be finite Borel partitions with
		$\mesh(\alpha_m)\to0$. If a compact set $L$ satisfies $h(T,L)>0$, then
		$\overline h_{\alpha_m}(L)>0$ for some $m$.
	\end{lemma}
	
	\begin{proof}
		Choose $\varepsilon>0$ at which the $\varepsilon$-scale spanning entropy of
		$L$ is positive, and then take $m$ with
		$\mesh(\alpha_m)<\varepsilon$. Each atom of
		$(\alpha_m)_0^{n-1}$ has $d_n$-diameter less than $\varepsilon$. Selecting
		one point from every name atom meeting $L$ therefore gives an
		$(n,\varepsilon)$-spanning set, whence
		\[
		\Nname_{\alpha_m}(n,L)\geq r_n(d,T,\varepsilon,L).
		\]
		The exponential limsup is positive.
	\end{proof}
	
	For $0<u<1$, let
	\[
	\Hb(u):=-u\log u-(1-u)\log(1-u)
	\]
	denote the binary entropy function. The following estimate is the usual
	Hamming-ball bound from information theory; see
	\cite{CoverThomas2006,Shannon1949}.
	
	\begin{lemma}
		\label{lem:hamming-bowen}
		Let $\beta=\{B_g:g\in G\}$ be a finite Borel partition. For every $g\in G$,
		let $C_g\subset B_g$ be a nonempty compact set, and put
		\[
		D:=X\setminus\bigcup_{g\in G}C_g,
		\quad
		\Delta:=\min_{g\neq g'}d(C_g,C_{g'})>0.
		\]
		Suppose that a compact set $L$ has the following property: for some
		$\theta\in(0,1/4)$ and $N$, every $x\in L$ and $n\geq N$ satisfy
		\[
		\#\{0\leq i<n:T^ix\in D\}\leq\theta n.
		\]
		Then
		\[
		h(T,L)
		\geq
		\overline h_\beta(L)-\Hb(2\theta)-2\theta\log|G|.
		\]
	\end{lemma}
	
	\begin{proof}
		Let $E_n$ be a maximal $(n,\Delta/3)$-separated subset of $L$, which Bowen balls cover $L$. If two points in one such ball are outside $D$ at time $i$
		but lie in different atoms of $\beta$, then their images lie in two distinct
		compact subsets and are at least $\Delta$ apart. On the other hand, both are
		within $\Delta/3$ of the ball centre in the Bowen metric, a contradiction.
		Thus their $\beta$-names differ in at most $2\theta n$ coordinates.
		
		A Hamming ball of radius $2\theta n$ in an alphabet of cardinality $|G|$
		contains at most
		\[
		\sum_{j\leq2\theta n}\binom nj|G|^j
		\leq
		\exp\left(
		n\bigl[\Hb(2\theta)+2\theta\log|G|+\log(n+1)/n\bigr]
		\right)
		\]
		words. Therefore
		\[
		\Nname_\beta(n,L)
		\leq
		|E_n|
		\exp\left(
		n\bigl[\Hb(2\theta)+2\theta\log|G|+\log(n+1)/n\bigr]
		\right).
		\]
		Taking exponential growth rates gives the assertion.
	\end{proof}
	
	\subsection{Proof of the infinite-measure-entropy obstruction}
	\label{sec:infinite-proof}
	
    The relative Sinai theorem used below is the version stated in \cite[Lemma~4.2]{HYZ2014}, whose roots can be traced back to \cite{OrnsteinWeiss1975,Thouvenot1975,Kieffer1984}. We also use the
	Shannon--McMillan--Breiman and Birkhoff ergodic theorems in their standard finite
	partition forms; see \cite{Walters1982}.
	
	\begin{proof}[Proof of \cref{thm:infinite-main}]
		Fix $\mu\in \mathcal{M}^e(X,T)$ with $h_\mu(T)=+\infty$.  We begin by
		choosing the finite partitions that will provide, respectively, a fixed
		positive entropy scale and a sequence of spatial scales tending to zero.
		Since
		\[
		h_\mu(T)=\sup_{\alpha\in\cP_X}h_\mu(T,\alpha)=+\infty,
		\]
		choose a finite Borel partition $\alpha_1$ such that
		\[
		a:=h_\mu(T,\alpha_1)>0.
		\]
		Choose finite Borel partitions $\tau_m$ such that
		$\mesh(\tau_m)\to0$, and set
		\[
		\alpha_m:=\alpha_1\vee\tau_m,
		\]
		discarding empty atoms. Then $\mesh(\alpha_m)\to0$ and every $\alpha_m$
		refines $\alpha_1$. Write
		$k_m=\#\alpha_m$ and label the atoms of $\alpha_m$ by
		$\Z/k_m\Z$.

		We next construct a sequence of Bernoulli partitions.  Enumerate
		the pairs of positive integers as
		\[
		(m_t,q_t),\quad t\geq1,
		\]
		so that, for each fixed $m$, the corresponding values of $q_t$ are
		unbounded. Put
		\[
		G_t:=(\Z/k_{m_t}\Z)^{q_t}.
		\]
		We recursively construct finite partitions $\beta_t$. Suppose that
		$\beta_1,\ldots,\beta_{t-1}$ have already been chosen, and put
		\[
		\xi_t
		:=
		\alpha_1\vee
		\alpha_{m_1}\vee\cdots\vee\alpha_{m_t}\vee
		\beta_1\vee\cdots\vee\beta_{t-1},
		\quad
		\cF_t:=\xi_t^T.
		\]
		The entropy $h_\mu(T,\xi_t)$ is finite. Since the total measure entropy is
		infinite, there is a finite partition $\zeta_t$ such that
		\[
		h_\mu(T,\xi_t\vee\zeta_t)-h_\mu(T,\xi_t)
		\geq\log|G_t|.
		\]
		Apply the relative Sinai theorem to
		$\alpha=\xi_t$, $\gamma=\xi_t\vee\zeta_t$, and the uniform distribution on
		$G_t$. We obtain a partition $\beta_t$, labelled by $G_t$, such that
		\begin{enumerate}
			\item the process $(\beta_t(T^ix))_{i\in\Z}$ is uniform i.i.d.;
			\item $\beta_t^T$ is independent of $\cF_t$.
		\end{enumerate}

		Let
		$b_t(x)\in G_t$ be the label of the $\beta_t$-atom containing $x$, and
		let $a_m(x)\in\Z/k_m\Z$ be the label of the $\alpha_m$-atom containing $x$.
		Define
		\[
		A_{t,i}(x)
		:=
		\bigl(
		a_{m_t}(T^{q_ti}x),\ldots,
		a_{m_t}(T^{q_ti+q_t-1}x)
		\bigr)\in G_t
		\]
		and
		\[
		\rho_{t,i}(x):=b_t(T^ix)-A_{t,i}(x)\in G_t.
		\]
		Let
		\[
		R_t(x):=(\rho_{t,i}(x))_{i\in\Z},
		\quad
		\cR_t:=\sigma(R_t).
		\]
		Put
		\[
		S_t:=G_t^\Z,
		\quad
		B_t(x):=\bigl(b_t(T^ix)\bigr)_{i\in\Z}\in S_t.
		\]
		The process $A_t=(A_{t,i})_{i\in\Z}$ is $\cF_t$-measurable, whereas
		$B_t$ is independent of $\cF_t$ and has the uniform product distribution
		$\lambda_t$ on $S_t$.  Moreover, $R_t=B_t-A_t$, with subtraction taken
		coordinatewise.
		
		If $f$ is a bounded Borel
		function on $S_t$, define
		\[
		H(x):=\int_{S_t}f\bigl(b-A_t(x)\bigr)\,d\lambda_t(b).
		\]
		This is $\cF_t$-measurable.  For $D\in\cF_t$, let
		\[
		\nu_D(E):=\mu\bigl(D\cap\{A_t\in E\}\bigr),
		\quad E\in\cB(S_t).
		\]
		Independence of $B_t$ and $\cF_t$ gives, for Borel sets $C,E\subset S_t$,
		\[
		\mu\bigl(D\cap\{B_t\in C\}\cap\{A_t\in E\}\bigr)
		=
		\lambda_t(C)\nu_D(E).
		\]
		By the $\pi$--$\lambda$ theorem, the finite measure induced by
		$(B_t,A_t)$ on $D$ is therefore $\lambda_t\otimes\nu_D$, and Fubini's
		theorem yields
		\[
		\begin{aligned}
			\int_D f(B_t-A_t)\,d\mu
			&=
			\int_{S_t\times S_t}f(b-a)\,d\lambda_t(b)\,d\nu_D(a)\\
			&=
			\int_D\left[\int_{S_t}f\bigl(b-A_t(x)\bigr)\,d\lambda_t(b)\right]d\mu(x)
			=
			\int_DH\,d\mu.
		\end{aligned}
		\]
		Consequently,
		\[
		\mathbb E_\mu\bigl[f(R_t)\mid\cF_t\bigr](x)
		=
		\int_{S_t} f(b-A_t(x))\,d\lambda_t(b)
		=
		\int_{S_t} f(b)\,d\lambda_t(b).
		\]
		The last equality follows from translation invariance of the Haar
		measure $\lambda_t$.
		
		Hence, $\cR_t\perp\cF_t$. Indeed, for any $C\in\cB(S_t)$, apply the formula to $f=\mathbf 1_C$ to obtain
		\[
		\mathbb E_\mu\bigl[
		\mathbf 1_{\{R_t\in C\}}\mid\cF_t
		\bigr]
		=
		\lambda_t(C).
		\]
		Taking expectations shows that $\mu(R_t\in C)=\lambda_t(C)$.  Hence, for
		$D\in\cF_t$,
		\[
		\begin{aligned}
			\mu\bigl(\{R_t\in C\}\cap D\bigr)
			&=
			\int_D\mathbb E_\mu\bigl[
			\mathbf 1_{\{R_t\in C\}}\mid\cF_t
			\bigr]d\mu\\
			&=
			\lambda_t(C)\mu(D)
			=
			\mu(R_t\in C)\mu(D).
		\end{aligned}
		\]
		Since all events $\{R_t\in C\}$ generate $\cR_t=\sigma(R_t)$, this proves
		\begin{equation}
			\cR_t\perp\cF_t.
			\label{eq:residual-independence}
		\end{equation}
	    Put
		\[
		\mathcal A:=\alpha_1^T,
		\quad
		\mathcal J_{t-1}:=\bigvee_{s<t}\mathcal R_s.
		\]
		For $s<t$, the $\sigma$-algebra $\cF_t$ contains both
		$\alpha_{m_s}^T$ and $\beta_s^T$.  Since $R_s$ is measurable with respect
		to $\alpha_{m_s}^T\vee\beta_s^T$, it follows that
		\[
		\mathcal A\vee\mathcal J_{t-1}
		\subset\mathcal F_t.
		\]
		Consequently, \eqref{eq:residual-independence} gives the stronger recursive
		relation
		\begin{equation}
			\mathcal R_t
			\perp
			\bigl(\mathcal A\vee\mathcal J_{t-1}\bigr).
			\label{eq:recursive-residual-independence}
		\end{equation}
		If $D\in\mathcal A$ and $E_s\in\mathcal R_s$ for $1\leq s\leq t$,
		repeated application of \eqref{eq:recursive-residual-independence} yields
		\begin{align}
			\mu\left(D\cap\bigcap_{s=1}^tE_s\right)
			&=
			\mu(E_t)\,
			\mu\left(D\cap\bigcap_{s=1}^{t-1}E_s\right)\notag\\
			&=
			\mu(D)\prod_{s=1}^t\mu(E_s).
			\label{eq:finite-residual-cylinder}
		\end{align}
		Taking $D=X$ shows that
		\[
		\mu\left(\bigcap_{s=1}^tE_s\right)
		=
		\prod_{s=1}^t\mu(E_s).
		\]
		Thus every finite $R_s$-cylinder is independent of $\mathcal A$.
		
		For completeness, let $\mathscr P$ be the family consisting of $X$ and all
		finite intersections
		\[
		\bigcap_{j=1}^m E_{t_j},
		\quad E_{t_j}\in\mathcal R_{t_j}.
		\]
		Then $\mathscr P$ is a $\pi$-system and
		$\sigma(\mathscr P)=\bigvee_{t\geq1}\mathcal R_t$. For fixed \(D\in\mathcal F_t\), let
		\[
		\mathscr L_D
		:=
		\{E\in\mathcal B(S_t\times S_t):
		\nu_D^{(B,A)}(E)=(\lambda_t\otimes\nu_D)(E)\}.
		\]
		Then \(\mathscr L_D\) is a \(\lambda\)-system.
		\eqref{eq:finite-residual-cylinder} shows that
		$\mathscr P\subset\mathscr L_D$.  Dynkin's $\pi$--$\lambda$ theorem therefore
		implies $\sigma(\mathscr P)\subset\mathscr L_D$.  Since this holds for every
		$D\in\mathcal A$, we have proved
		\begin{equation}
			\join_{t\geq1}\cR_t =:\cR
			\perp
			\alpha_1^T.
			\label{eq:total-independence}
		\end{equation}
		We may note that this $\sigma$-algebra $\cR$ need not be $T$-invariant; however, only ordinary
		disintegration over the associated Borel map will be used.

		Choose
		$\theta_t\in(0,1/4)$ so small that
		\begin{equation}
			e_t:=\Hb(2\theta_t)+2\theta_t\log|G_t|<1.
			\label{eq:hamming-error}
		\end{equation}
		Write $\beta_t=\{B_{t,g}:g\in G_t\}$.  Every atom has positive measure,
		because the one-coordinate distribution is uniform on
		$G_t$.  By inner regularity, choose nonempty compact cores
		$C_{t,g}\subset B_{t,g}$ sufficiently large that, defining
		\[
		D_t:=X\setminus\bigcup_{g\in G_t}C_{t,g},
		\]
		we have
		\begin{equation}
			\mu(D_t)<\theta_t/2.
			\label{eq:bad-set-small}
		\end{equation}
		The sets $C_{t,g}$ are pairwise disjoint and their minimum mutual distance is positive, i.e.
		\begin{equation}
			\Delta_t
			:=
			\min_{g\neq g'}d(C_{t,g},C_{t,g'})>0.
			\label{eq:compact-core-gap}
		\end{equation}
		
		Let $R=(R_t)_{t\geq1}: X \to Y:= \prod_{t\geq 1} S_t$ and disintegrate $\mu$ over this Borel map,
		\[
		\mu=\int\eta_r\,d\nu(r), \quad \text{where}~\nu=R_* \mu.
		\]
		For $A\in\alpha_1^T$ and a Borel set $B \subset R(X)$,
		\eqref{eq:total-independence} gives
		\[
		\int_B\eta_r(A)\,d\nu(r)
		=
		\mu(A\cap R^{-1}(B))
		=
		\mu(A)\nu(B).
		\]
		Since $\alpha_1^T$ is countably generated, first applying this identity to a
		countable generating algebra and then using a monotone-class argument shows
		that, for $\nu$-a.e.  $r$,
		\begin{equation}
			\eta_r(A)=\mu(A)
			\quad(A\in\alpha_1^T),
			\quad
			\eta_r(R^{-1}(r))=1.
			\label{eq:fiber-properties}
		\end{equation}
		
		By Birkhoff's theorem and \eqref{eq:bad-set-small}, for each $t$ and
		$\mu$-a.e. $x$,
		\[
		\frac1n\sum_{i=0}^{n-1}\mathbf1_{D_t}(T^ix)\leq\theta_t
		\]
		for all sufficiently large $n$.  Let $Y_t$ be the set of points for which
		this eventual estimate holds.  Then $\mu(Y_t)=1$ for every $t \geq 1$, and hence
		$\mu(\bigcap_{t\geq 1} Y_t)=1$.  Disintegration shows that
		$\eta_r(\bigcap_{t\geq1}Y_t)=1$ for $\nu$-a.e. $r$.  Fix one $r$ for which
		this conclusion and both assertions in \eqref{eq:fiber-properties} hold.
		
		Define
		\[
		E_{t,N}
		:=
		\left\{
		x:
		\frac1n\sum_{i=0}^{n-1}\mathbf1_{D_t}(T^ix)
		\leq\theta_t
		\text{ for every }n\geq N
		\right\}.
		\]
		For the fixed $r$, the sets $E_{t,N}$ increase to $Y_t$ as $N\to\infty$.
		Choose $\varepsilon_t>0$ such that
		$\sum_{t\geq1}\varepsilon_t<1/8$, and then choose $N_t$ so that
		\[
		\eta_r(E_{t,N_t})>1-\varepsilon_t.
		\]
		Using \eqref{eq:fiber-properties} gives
		\[
		\eta_r\left(
		R^{-1}(r)\cap\bigcap_{t\geq1}E_{t,N_t}
		\right)
		>\frac78.
		\]
		By inner regularity of
		$\eta_r$, there exists a compact set $K \subset R^{-1}(r)\cap\bigcap_{t\geq1}E_{t,N_t}$ such that
		\begin{equation}
			\eta_r(K)>\frac34.
			\label{eq:K-positive}
		\end{equation}
		Thus, for every $x\in K$ and every $n\geq N_t$,
		\begin{equation}
			\#\{0\leq i<n:T^ix\in D_t\}
			\leq\theta_tn.
			\label{eq:uniform-bad-frequency}
		\end{equation}

		Next, we show that $K$ is precisely the set required in \cref{thm:infinite-main}. Let
		$\cA_n$ be the union of all atoms of $(\alpha_1)_0^{n-1}$ that meet
		$K$. Then $\cA_n\in\alpha_1^T$, and
		\eqref{eq:fiber-properties}--\eqref{eq:K-positive} give
		\begin{equation}
			\mu(\cA_n)=\eta_r(\cA_n)\geq\eta_r(K)>\frac34.
			\label{eq:A-large}
		\end{equation}
		Fix $\varepsilon\in(0,a)$, and let $V_{n,\varepsilon}$ be the union of the
		atoms $A\in(\alpha_1)_0^{n-1}$ satisfying
		$\mu(A)\leq e^{-n(a-\varepsilon)}$.  The
		Shannon--McMillan--Breiman theorem gives
		$\mu(V_{n,\varepsilon})\to1$.  Thus, for all sufficiently large $n$,
		\[
		\mu(\cA_n\cap V_{n,\varepsilon})
		>\frac12.
		\]
		Every atom occurring in this intersection has measure at most
		$e^{-n(a-\varepsilon)}$ and meets $K$.  Hence at least
		$\frac12e^{n(a-\varepsilon)}$ distinct names of length $n$ occur in $K$, so
		\[
		\Nname_{\alpha_1}(n,K)
		\geq\frac12e^{n(a-\varepsilon)}.
		\]
		Letting $\varepsilon\downarrow0$ yields
		\begin{equation}
			\overline h_{\alpha_1}(K)\geq a>0.
			\label{eq:K-positive-name-entropy}
		\end{equation}

		Fix
		$t$ and a compact set $L\subset K$.  Since $K\subset R^{-1}(r)$, for each
		$i\in\Z$ there is a fixed element $r_{t,i}\in G_t$ such that
		$\rho_{t,i}(x)=r_{t,i}$ for every $x\in K$.  Therefore
		\[
		A_{t,i}(x)=b_t(T^ix)-r_{t,i}
		\quad(x\in K).
		\]
		Translation by the fixed element $r_{t,i}$ is a bijection of $G_t$, so the
		list of key symbols
		\[
		\bigl(b_t(x),b_t(Tx),\ldots,b_t(T^{n-1}x)\bigr)
		\]
		determines, and is determined by, the list of blocks
		$\bigl(A_{t,0}(x),\ldots,A_{t,n-1}(x)\bigr)$.  Concatenating these $n$
		blocks gives exactly the $\alpha_{m_t}$-name of $x$ at times
		$0,1,\ldots,q_tn-1$.  Hence the two collections of names occurring in
		$L$ are in bijection, and
		\begin{equation}
			\Nname_{\beta_t}(n,L)
			=
			\Nname_{\alpha_{m_t}}(q_tn,L).
			\label{eq:exact-name-count}
		\end{equation}
		By \cref{lem:subsequence},
		\begin{equation}
			\overline h_{\beta_t}(L)
			=
			q_t\,\overline h_{\alpha_{m_t}}(L).
			\label{eq:block-entropy-scaling}
		\end{equation}
		Applying \cref{lem:hamming-bowen} with
		\eqref{eq:hamming-error}, \eqref{eq:compact-core-gap}, and
		\eqref{eq:uniform-bad-frequency}, we obtain
		\begin{equation}
			h(T,L)
			\geq
			\overline h_{\beta_t}(L)-e_t
			\geq
			q_t\,\overline h_{\alpha_{m_t}}(L)-1.
			\label{eq:entropy-amplification}
		\end{equation}
		
		When we take the indices $t$ with $m_t=1$, the corresponding $q_t$ are unbounded, so
		\eqref{eq:K-positive-name-entropy} and
		\eqref{eq:entropy-amplification} imply
		\begin{equation}
			h(T,K)=+\infty.
			\label{eq:K-infinite}
		\end{equation}
		Now let $L\subset K$ be compact with $h(T,L)>0$. By \cref{lem:mesh}, there
		is an $m$ such that $\overline h_{\alpha_m}(L)>0$. Keeping this $m$ fixed
		and taking the corresponding unbounded values of $q_t$ in
		\eqref{eq:entropy-amplification}, we get $h(T,L)=+\infty$. Therefore
		\[
		h(T,K)=+\infty,
		\quad 
		h(T,L)\in\{0,+\infty\}
		\]
		for every compact $L\subset K$. So $K$ is not lowerable. By
		\cref{def:lowerable}, $(X,T)$ is not hereditarily lowerable.
	\end{proof}
		
	\begin{corollary}
		\label{cor:conjecture}
		\Cref{conj:infinite} is true.
	\end{corollary}
	
	\begin{remark}
		The above set $K$, constructed in \cref{thm:infinite-main}, has positive measure for each conditional measure $\eta_r$, but is typically $\mu$-null. This is no accident. Indeed, if a compact set $C$ has $\mu(C)>0$, then the relative entropy extraction in \cite[Proposition 4.3 and Lemma 4.1]{HYZ2014} allows one to realize any prescribed finite entropy value within $C$. Consequently, a compact set of positive $\mu$-measure cannot serve as a certificate for non-lowerability; the proof must confine all amplification relations to the same $\mu$-null conditional fibre.
	\end{remark}
	\begin{sloppypar}
		
	\end{sloppypar}
	
\end{document}